\documentclass[11pt]{article}

\usepackage{geometry}[1in]

\usepackage{amsmath}
\usepackage{amsthm}
\usepackage{amsfonts}
\usepackage{amsthm}
\usepackage{amsrefs}

\usepackage{tikz}
\usepackage{thm-restate}
\usepackage{subcaption}

\usepackage{cleveref}

\newtheorem{thm}{Theorem}
\crefname{thm}{Theorem}{Theorems}
\Crefname{thm}{Theorem}{Theorems}

\newtheorem{lem}[thm]{Lemma}
\crefname{lem}{Lemma}{Lemmas}
\Crefname{lem}{Lemma}{Lemmas}

\newtheorem{prop}[thm]{Proposition}
\crefname{prop}{Proposition}{Propositions}
\Crefname{prop}{Proposition}{Propositions}

\crefname{cor}{Corollary}{Corollaries}
\Crefname{cor}{Corollary}{Corollaries}

\newtheorem{conj}[thm]{Conjecture}
\crefname{conj}{Conjecture}{Conjectures}
\Crefname{conj}{Conjecture}{Conjectures}

\newtheorem{qn}[thm]{Question}
\crefname{qn}{Question}{Questions}
\Crefname{qn}{Question}{Questions}

\newtheorem{obs}[thm]{Observation}

\theoremstyle{definition}

\newcommand{\No}{\mathbb{Z}^{\ge0}}
\newcommand{\Z}{\mathbb{Z}}

\newcommand{\R}{\mathbb{R}}
\newcommand{\F}{\mathcal{F}}
\newcommand{\es}{\emptyset}

\newcommand{\rpart}[1]{$#1$-partition}
\newcommand{\modpart}[1]{$(1\text{ mod }#1)$-partition}

\begin{document}

\title{Decomposing the vertex set of a hypercube into isomorphic subgraphs}

\author{%
    Vytautas Gruslys%
    \thanks{%
        Department of Pure Mathematics and Mathematical Statistics,
        University of Cambridge,
        Wilberforce Road,
        CB3\,0WA Cambridge,
        United Kingdom;
        e-mail: \texttt{v.gruslys@cam.ac.uk}\,.
    }
}

\maketitle

\begin{abstract}
    Let $G$ be an induced subgraph of the hypercube $Q_k$ for some $k$. We show
    that if $|G|$ is a power of $2$ then, for sufficiciently large $n$, the
    vertex set of $Q_n$ can be partitioned into induced copies of $G$. This
    answers a question of Offner.
    In fact, we prove a stronger statement: if $X$ is a subset of $\{0,1\}^k$
    for some $k$ and if $|X|$ is a power of $2$, then, for sufficiently large
    $n$, $\{0,1\}^n$ can be partitioned into isometric copies of $X$.
\end{abstract}

\section{Introduction} \label{sec:introduction}

A famous theorem of Wilson~\cite{Wilson1976} states that, for any finite graph
$H$ and for any sufficiently large integer $n$ which satisfies certain
divisibility conditions, the edges of the complete graph $K_n$ can be covered by
disjoint copies of $H$.  Such a cover is called an \emph{$H$-decomposition} of
$K_n$. The divisibility conditions required by Wilson's theorem are obviously
necessary for an $H$-decomposition of $K_n$ to exist: $\binom{n}{2}$ must be
divisible by $e(H)$ and $n-1$ must be divisible by the highest common factor of
the degrees of the vertices of $H$. Therefore, as long as we are only interested
in large $n$, Wilson's theorem tells us exactly when $K_n$ admits an
$H$-decomposition.  On the other hand, the general question of determining
whether an arbitrary graph $G$ has a $H$-decomposition is very difficult, and
various special cases of this question have attracted significant attention.

In this paper we examine a related question: we are concerned with partitioning
the vertices -- not edges -- of a given graph $G$ into copies of $H$. More
precisely, for finite graphs $G,H$, we say that a set $A \subset V(G)$ is an
\emph{$H$-set} if the induced subgraph $G[A]$ is isomorphic to $H$. We consider
the following question: can $V(G)$ be partitioned into $H$-sets?

In contrast to Wilson's theorem, this question is not interesting in the case
where $G$ is a complete graph: obviously, $V(K_n)$ can be partitioned into
$H$-sets if and only if $H = K_m$ where $m$ divides $n$. Instead, we focus on
the case where $G$ is the hypercube $Q_n$, that is, the graph with vertex set
$\{0,1\}^n$ where two $n$-tuples are adjacent if and only if they differ in
precisely one entry.

Let $H$ be a finite graph and let $n$ be large. Can we quickly determine whether
$V(Q_n)$ can be partitioned into $H$-sets? Of course, there is an obvious
necessary divisibility condition: $|H|$ must be a power of $2$. Moreover, this
condition alone is not sufficient because $H$ may not be isomorphic to any
induced subgraph of any hypercube $Q_n$. For example, $H$ could be a
non-bipartite graph or, say, it could be a bipartite graph, of size a power of
$2$, that contains $K_{3,2}$ as a subgraph.  Note $K_{3,2}$ is not a subgraph of
any $Q_n$ since any two vertices that are distance $2$ apart in $Q_n$ are joined
by precisely two paths of length $2$.  Therefore, we should require $H$ to be an
induced subgraph of some hypercube.  Offner~\cite{Offner2014} considered this
problem in connection with coding theory. He asked if this condition together
with the divisibility condition is sufficient.

\begin{qn}[Offner] \label{qn:offner}
    Let $H$ be an induced subgraph of $Q_k$ for some $k$ and suppose that $|H|$
    is a power of $2$. Must it be true that, for any sufficiently large $n$,
    $V(Q_n)$ can be partitioned into $H$-sets?
\end{qn}

This question bears resemblence to the celebrated work of Hamming~%
\cite{Hamming1950} on error-correcting codes. Indeed, a perfect
single-error-correcting code is a partition of $V(Q_n)$ into $K_{1,n}$-sets.
Hamming showed that such a partition exists if and only if a natural
divisibility condition is satisfied, namely, if $n = 2^r - 1$ for some $r$. Much
later, Rogers~\cite{RogersUnpublished,Ramras1992} asked if it is possible to
partition the vertices of $Q_n$ into antipodal paths, subject to the same
divisibility condition. Here an \emph{antipodal path} is a path of length $n$
which starts and ends at two diagonally opposite vertices of $Q_n$. Rogers'
question was answered by Ramras~\cite{Ramras1992}, who proved the following more
general result: if $n = 2^r - 1$ and if $T$ is a tree on $n+1$ vertices which is
an induced subgraph of $Q_n$, then $V(Q_n)$ can be partitioned into isometric
copies of $T$.

On the other hand, there is an important difference between Offner's question
and the work of Hamming and Ramras: in Offner's question $H$ is fixed and $n$
can be taken to be large.  In fact, Offner's question is more closely related to
the following two conjectures coming from different areas of combinatorics. One
was proposed by Chalcraft~\cite{ChalcraftI, ChalcraftII}.

\begin{conj}[Chalcraft] \label{conj:chalcraft}
    Let $T$ be a non-empty finite subset of $\Z^k$ for some $k$, where $\Z^k$ is
    treated as a subspace of the metric space $\R^k$. Then, for
    sufficiently large $n$, the space $\Z^n$ can be partitioned into isometric
    copies of $T$.
\end{conj}

Another related conjecture was proposed by Lonc~\cite{Lonc1991}.

\begin{conj}[Lonc] \label{conj:lonc}
    Let $P$ be a poset with a greatest and a least element. If $|P|$ is a power
    of $2$, then, for sufficiently large $n$, the Boolean lattice $2^{[n]}$ can
    be partitioned into sets, each of which induces a poset isomorphic to $P$.
\end{conj}

These conjectures were recently solved: Gruslys, Leader and Tan~%
\cite{Gruslys2016I} confirmed Chalcraft's conjecture and Gruslys, Leader and
Tomon~\cite{Gruslys2016II} confirmed Lonc's conjecture.  In this paper we
combine new ideas with tools developed by these authors to give a positive
answer to Offner's question.

\begin{restatable}{thm}{main} \label{thm:main}
    Let $H$ be an induced sugraph of $Q_k$ for some $k$. If $|H|$ is a power of
    $2$, then there exists a positive integer $n$ such that the vertices of
    $Q_n$ can be partitioned into $H$-sets.
\end{restatable}

Of course, if the result holds for $n$, then it holds for all $n' \ge n$.
Therefore, \Cref{thm:main} answers \Cref{qn:offner}.

\section{Overview of the proof} \label{sec:overview}

It turns out that, in order to prove \Cref{thm:main}, it is convenient to view
the hypercube $Q_n$ as the metric space $\{0,1\}^n$ where the distance between
any two points $x,y \in \{0,1\}^n$, denoted $d(x,y)$, is equal to the number of
entries where $x$ and $y$ are different. With this definition, $d(x,y)$ equals
$1$ if and only if $x$ and $y$ are adjacent vertices of $Q_n$. If $H$ is an
induced subgraph of $Q_k$, then we can identify $H$ with a subset of
$\{0,1\}^k$. For any $n \ge k$, we say that a set $X \subset \{0,1\}^n$ is an
\emph{isometric copy of $H$} if there exists an isometry $\phi : \{0,1\}^k \to
\{0,1\}^n$ which maps $H$ to $X$. Clearly, any isometric copy of $H$ in
$\{0,1\}^n$ is an $H$-set, but an $H$-set need not be an isometric copy of $H$.

We deduce \Cref{thm:main} from the following slightly stronger result.

\begin{restatable}{thm}{main-isometric} \label{thm:main-isometric}
    Let $X$ be a subset of $\{0,1\}^k$ for some $k$. If $|X|$ is a power of $2$,
    then there exists a positive integer $n$ such that $\{0,1\}^n$ can be
    partitioned into isometric copies of $X$.
\end{restatable}

A major tool in our proof of \Cref{thm:main-isometric} is a theorem of Gruslys,
Leader and Tomon~\cite{Gruslys2016II}. Roughly speaking, their theorem says
that, if we are trying to partition an arbitrarily large power $A^n$ of some set
$A$ into copies of some given set (which is exactly what we are doing here),
then it is enough to construct two specific covers of a large power of $A$,
called an `\rpart{r}' and a `\modpart{r}'.  We will now define these covers and
we will see that it is easier to construct them than to directly construct a
partition of $\{0,1\}^n$. We will state the aforementioned theorem of Gruslys,
Leader and Tomon (\Cref{thm:blackbox}) after we have given the necessary
definitions.

Let $\F$ be a family of subsets of a set $S$. A \emph{weight function on $\F$}
is an assignment of non-negative integer weights to the members of $\F$. For a
weight function $w : \F \to \No$ and an element $x \in S$, the
\emph{multiplicity of $x$ for $w$} is defined to be the sum of weights
assigned to the members of $\F$ that contain $x$.

Let $\F$ and $S$ be as above and let $r$ be a positive integer. We say that
\emph{$\F$ contains an \rpart{r} of $S$} if there exists a weight function on
$\F$ for which every element of $S$ has multiplicity $r$. Moreover, we say
that \emph{$\F$ contains a \modpart{r} of $S$} if there exists a weight function
on $\F$ for which the multiplicity of every element of $S$ is congruent to $1
\pmod{r}$; it is not required that all elements of $S$ have the same
multiplicity as long as they all satisfy the required congruence.

Trivially, $\F$ contains a \rpart{1} of $S$ if and only if $S$ can be
partitioned into members of $\F$. Furthermore, if $\F$ contains a \rpart{1} of
$S$, then $\F$ also contains an \rpart{r} and a \modpart{r} of $S$, for any
positive integer $r$. Therefore, the property of containing an \rpart{r} and a
\modpart{r} for some $r$ is weaker that that of containing a genuine partition.
However, we will be able to apply the aforementioned theorem of Gruslys, Leader
and Tomon to obtain the stronger property from the weaker one.  The statement of
this theorem requires one more technical definition, and we postpone it until
the end of the section. Instead, we will now discuss how to construct an
\rpart{r} and a \modpart{r} of $\{0,1\}^n$ into isometric copies of some $X
\subset \{0,1\}^k$ whose size is a power of $2$. First, we have to decide what
value of $r$ to use.  It turns out that the right choice is $r = |X|$.

\begin{obs} \label{obs:rpart}
    Let $X$ be a non-empty subset of $\{0,1\}^k$ for some positive integer $k$.
    Then, for any $n \ge k$, the family of isometric copies of $X$ in
    $\{0,1\}^n$ contains a \rpart{|X|} of $\{0,1\}^n$.
\end{obs}

\begin{proof}
    Let $n \ge k$ be given. We fix one isometric copy of $X$ in $\{0,1\}^n$,
    which we denote by $Y$.  Under addition modulo $2$, for any $p \in
    \{0,1\}^n$, the set $Y+p = \{y+p : y \in Y\}$ is a subset of $\{0,1\}^n$.
    Moreover, it is an isometric copy of $X$.

    By symmetry, all elements of $\{0,1\}^n$ are contained in $Y+p$ for the same
    number of choices of $p$. By double counting, this number must equal
    $2^n|Y|/2^n = |X|$. Therefore, the sets $Y+p$, where $p \in \{0,1\}^n$, form
    a \rpart{|X|} of $\{0,1\}^n$.
\end{proof}

Constructing a \modpart{|X|} is rather more difficult, but also possible.

\begin{lem} \label{lem:modpart}
    Let $X$ be a non-empty subset of $\{0,1\}^k$ for some positive integer $k$,
    and let $r$ be a power of $2$. Then there exists an integer $n \ge k$ such
    that the family of all isometric copies of $X$ in $\{0,1\}^n$ contains a
    \modpart{r} of $\{0,1\}^n$.
\end{lem}

Although we are only going to use this lemma with $r = |X|$, we state it with
$r$ being any power of $2$. This small detail will allow us to prove this lemma
by induction, which we do in \Cref{sec:proof}.

We now turn to giving the final definitions needed for the statement of
\Cref{thm:blackbox}. Let $S$ be a set. Here and in the rest of the paper, for
any non-negative integers $m,n$, we identify $S^m \times S^n$ with $S^{m+n}$.
Therefore, for any $x \in S^m, y \in S^n$, we treat $(x,y)$ as an element of
$S^{m+n}$. Furthermore, for any set $X \subset S^n$ and any permutation $\pi :
\{1,\dotsc,n\} \to \{1,\dotsc,n\}$, we define $\pi(X)$ to be the image of $X$
after permuting the coordinates according to $\pi$. In other words, $\pi(X) =
\{(x_{\pi(1)}, \dotsc, x_{\pi(n)}) : (x_1, \dotsc, x_n) \in X\}$. Finally, for
any sets $Y \subset S^m, Z \subset S^n$ with $m \le n$, we say that $Z$ is a
\emph{copy} of $Y$ if $Z = \pi(Y \times \{z\})$ for some $z \in S^{n-m}$ and
some permutation $\pi : \{1,\dotsc,n\} \to \{1,\dotsc,n\}$.

\begin{thm}[Gruslys, Leader and Tomon~\cite{Gruslys2016II}] \label{thm:blackbox}
    Let $\F$ be a family of subsets of a finite set $S$. If, for some positive
    integer $r$, $\F$ contains an \rpart{r} and a \modpart{r} of $S$, then there
    exists a positive integer $n$ such that $S^n$ can be partitioned into copies
    of members of $\F$.
\end{thm}

\Cref{thm:blackbox} is the only result that we use without proof. We will now
explain how \Cref{obs:rpart,lem:modpart,thm:blackbox} imply
\Cref{thm:main-isometric}.

\begin{proof}[Proof of \Cref{thm:main-isometric}.]
    Let $X$ be a subset of $\{0,1\}^k$ such that $|X|$ is a power of $2$. It
    follows from \Cref{lem:modpart} that there exists a positive integer $m \ge
    k$ such that the family of isometric copies of $X$ in $\{0,1\}^m$ contains a
    \modpart{|X|} of $\{0,1\}^m$. By \Cref{obs:rpart}, the family of isometric
    copies of $X$ in $\{0,1\}^m$ also contains a \rpart{|X|} of $\{0,1\}^m$.
    Therefore, it follows from \Cref{thm:blackbox} with $S = \{0,1\}^m$ that
    there exists a positive integer $n$ such that $\{0,1\}^{mn}$ can be
    partitioned into copies of sets which are isometric copies of $X$ in
    $\{0,1\}^m$. However, a copy of an isometric copy of $X$ is itself an
    isometric copy of $X$, so we are done.
\end{proof}

\section{Constructing a \modpart{r} of $\{0,1\}^n$} \label{sec:proof}

Here we prove \Cref{lem:modpart}. This section is is the heart of the paper: it
is the key new ingredient beyond the ideas of \cite{Gruslys2016I} and
\cite{Gruslys2016II}.  First, we introduce some convenient notation. For any set
$A \subset \{0,1\}^n$, we define
\begin{align*}
    A_+ &= \left\{ a \in \{0,1\}^{n-1} : (a, 1) \in A \right\}, \\
    A_- &= \left\{ a \in \{0,1\}^{n-1} : (a, 0) \in A \right\}.
\end{align*}

\begin{proof}[Proof of \Cref{lem:modpart}.]

    Fix $r = 2^d$. We will use induction on $k$. If $k = 1$, then $X$ is either
    a single point or the whole $\{0,1\}$, and so the conclusion holds with $n =
    1$.
    
    We now suppose that $k \ge 2$. At least one of the sets $X_+$ and $X_-$ is
    not empty, so we may assume without loss of generality that $X_- \neq \es$.
    Since $X_-$ is a subset of $\{0,1\}^{k-1}$, the induction hypothesis implies
    the existence of a positive integer $m$ such that the family of isometric
    copies of $X_-$ in $\{0,1\}^m$ contains a \modpart{r} of $\{0,1\}^m$.
    Moreover, we note that, for every set $A \subset \{0,1\}^m$ which is an
    isometric copy of $X_-$, there exists a set $B \subset \{0,1\}^{m+1}$ which
    is an isometric copy of $X$ and which satisfies $B_- = A$. Therefore, it is
    possible to define a weight function on the family of isometric copies of
    $X$ in $\{0,1\}^{m+1}$ in such a way that the multiplicity of every element
    of $\{0,1\}^m \times \{0\}$ is congruent to $1 \pmod{r}$. We do not impose
    any conditions on the multiplicities of elements of $\{0,1\}^m \times
    \{1\}$. For convenience, we denote that the multiplicity of any $x \in
    \{0,1\}^m \times \{1\}$ is congruent to $f(x) \pmod{r}$.

    We will prove that the conclusion of \Cref{lem:modpart} holds with $n = m +
    d + 1$. Let $x, y \in \{0,1\}^{d+1}$ be two elements that differ in exactly
    two entries. There exists an element $z \in \{0,1\}^{d+1}$ that differs from
    both $x$ and $y$ in exactly one entry. Then $\{0,1\}^m \times \{x,z\}$ is an
    isometric copy of $\{0,1\}^{m+1}$, while $\{0,1\}^m \times \{x\}$ and
    $\{0,1\}^m \times \{z\}$ are isometric copies of $\{0,1\}^m$. Therefore,
    there exists an isometry $\phi : \{0,1\}^m \times \{x,z\} \to \{0,1\}^{m+1}$
    which maps $\{0,1\}^m \times \{x\}$ to $\{0,1\}^m \times \{0\}$ and
    $\{0,1\}^m \times \{z\}$ to $\{0,1\}^m \times \{1\}$. Hence, it is possible
    to assign integer weights to the isometric copies of $X$ in $\{0,1\}^m
    \times \{x,z\}$ so that the multiplicity of every element of $\{0,1\}^m
    \times \{x\}$ is congruent to $1 \pmod{r}$, and the multipliticity of any $p
    \in \{0,1\}^m \times \{z\}$ is congruent to $f(\phi(p)) \pmod{r}$. We denote
    the resulting weight function by $w'$.

    The restriction of $\phi$ to $\{0,1\}^m \times \{z\}$ maps this set
    isometrically onto $\{0,1\}^m \times \{1\}$. This map extends to an isometry
    $\{0,1\}^m \times \{y, z\} \to \{0,1\}^{m+1}$. Therefore, we can assign
    integer weights to the isometric copies of $X$ in $\{0,1\}^m \times \{y,z\}$
    in such a way that every element of $\{0,1\}^m \times \{y\}$ has
    multiplicity congruent to $1 \pmod{r}$, and any $p \in \{0,1\}^k \times
    \{z\}$ has multiplicity congruent to $f(\phi(p)) \pmod{r}$. We denote the
    resulting weight function by $w''$.

    Although, technically, the weight functions $w', w''$ are only defined on
    isometric copies of $X$ in, respectively, $\{0,1\}^m \times \{x,z\}$ and
    $\{0,1\}^m \times \{y,z\}$, we may suppose that they are defined and equal
    to $0$ on the other isometric copies of $X$ in $\{0,1\}^n$. Then
    $w'+(r-1)w''$, which we denote by $w_{x,y}$, is a weight function on the
    family of all isometric copies of $X$ in $\{0,1\}^n$. Moreover, for any $p
    \in \{0,1\}^n$, the multiplicity of $p$ for $w_{x,y}$ is congruent to
    $$
        \begin{cases}
            1 \pmod{r} & \text{if } p \in \{0,1\}^m \times \{x\}, \\
            -1 \pmod{r} & \text{if } p \in \{0,1\}^m \times \{y\}, \\
            0 \pmod{r} & \text{otherwise}.
        \end{cases}
    $$

    The existence of the weight functions $w_{x,y}$ simplifies our problem in
    the following way. Let us view $\{0,1\}^n$ as the product set $\{0,1\}^m
    \times \{0,1\}^{d+1}$. Given two elements $x, y \in \{0,1\}^{d+1}$ with
    $d(x,y) = 2$, we identify the pair $(x,y)$ with both the directed edge
    $\overrightarrow{xy}$ on $\{0,1\}^{d+1}$ and the weight function $w_{x,y}$.
    Now, our aim is to find a family (allowing repetitions) of directed edges on
    $\{0,1\}^{d+1}$, whose every member joins two elements of $\{0,1\}^{d+1}$
    that are distance $2$ apart, and such that for any $v \in \{0,1\}^{d+1}$ the
    difference between the in-degree and out-degree of $v$ is congruent to $1
    \pmod{p}$. Indeed, such a family of directed edges corresponds to a weight
    function for which every element of $\{0,1\}^n$ has multiplicity congruent
    to $1 \pmod{r}$.

    We will now construct a family of directed edges with the desired
    properties. Fix vertices $x^\ast = (0,\dotsc,0) \in \{0,1\}^{d+1}$ and
    $y^\ast = (1,0,\dotsc,0) \in \{0,1\}^{d+1}$. Note that, for any vertex $v
    \in \{0,1\}^{d+1}$, there exists a directed path starting from $x^\ast$ or
    $y^\ast$ and ending at $v$ with the property that any two consecutive
    vertices on this path differ in exactly two entries. Such a path increases
    the difference between the in-degree and the out-degree of $v$ by $1$,
    decreases this parameter of its starting point ($x^\ast$ or $y^\ast$) by $1$
    and does not change the value of this parameter for any other vertex. Now,
    for any vertex $v \in \{0,1\}^{d+1} \setminus \{x^\ast\}$ with an even
    number of $1$'s, select one such path from $x^\ast$ to $v$. Similarly, for
    any $v \in \{0,1\}^{d+1} \setminus \{y^\ast\}$ with an odd number of $1$'s,
    select one such path from $y^\ast$ to $v$. Let us combine all of these paths
    together to obtain a family of directed edges. It is clear that for any $v
    \in \{0,1\}^{d+1} \setminus \{x^\ast, y^\ast\}$ the difference between the
    in-degree and the out-degree of $v$ is equal to $1$. Moreover, excluding
    $x^\ast$, there are $2^d - 1$ vertices in $\{0,1\}^{d+1}$ with an even
    number of $1$'s. Therefore, the difference between the in-degree and the
    out-degree of $x^\ast$ is $-(2^d - 1) \equiv 1 \pmod{r}$. Similarly, the
    difference between the in-degree and the out-degree of $y^\ast$ is also
    congruent to $1 \pmod{r}$. This finishes the proof.  \end{proof}

\section{Concluding remarks and open problems}

The statement of \Cref{thm:main-isometric} is very similar to that of
Chalcraft's conjecture. Indeed, the only difference is that, instead of an
infinite space $\Z^n$, here we are dealing with a finite hypercube $\{0,1\}^n$.
However, the results are, in fact, significantly different.

To illustrate this claim, we note that not every sensible finite version of
Chalcraft's conjecture is true.  First, there is the issue of choosing which
metric to use. In $\Z^n$ or in any hypercube $[\ell]^n$ there are at least two
natural choices of a metric: the Euclidean metric $\allowbreak d( (x_1, \dotsc,
x_n), (y_1, \dotsc, y_n) ) \allowbreak = \sqrt{\sum_{i=1}^n (x_i - y_i)^2}$ and
the graph metric $\sum_{i=1}^n |x_i - y_i|$. Chalcraft's conjecture (for $\Z^n$)
is true for both metrics.  \Cref{thm:main-isometric} (for $[2]^n$) is
independent of the choice of the metric, since if $X,Y \subset \{0,1\}^n$ are
isometric copies with respect to one of the metrics then they are also isometric
copies with respect to the other. However, the situation is different in
$[\ell]^n$ for $\ell \ge 3$: the obvious version of Chalcraft's conjecture is
false for $[\ell]^n$ with the Euclidean metric. For example, take $\ell = 5$ and
let $T \subset [5]^2$ be a plus-shaped set of size $5$, as shown in
\Cref{fig:plus}. Then, no matter what $n$ we choose, it is impossible to
partition $[5]^n$ into isometric copies of $T$ because the corners of $[5]^n$
cannot be covered. Similar counterexamples exist for all $\ell \ge 3$.

\begin{figure}[h]
    \begin{centering}
    \begin{tikzpicture}
        \fill (0,0) circle (2pt);
        \fill (-.5,0) circle (2pt);
        \fill (.5,0) circle (2pt);
        \fill (0,-.5) circle (2pt);
        \fill (0,.5) circle (2pt);
    \end{tikzpicture}
    \caption{The plus-shaped set $T$.}
    \label{fig:plus}
    \end{centering}
\end{figure}
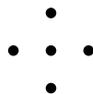

Second, the situation does not become trivial even if we choose the graph
metric.  It turns out that, with this metric, the obvious version of Chalcraft's
conjecture is true for $[\ell]^n$ where $\ell \ge 2$ is even. This fact can be
verified in a similar way to \Cref{thm:main-isometric}; essentially, the only
difference is that we have to partition $[\ell]^n$ into copies of $[2]^n$ before
we can apply \Cref{obs:rpart} (it is also important to note that $[\ell]^n$ can
be isometrically embedded into $[2]^m$ for sufficiently large $m$). However, the
corresponding conjecture would be false for $[\ell]^n$ where $\ell \ge 3$ is
odd. Indeed, we will demostrate that even the corresponding version of the
weaker \Cref{thm:main} is false.

We define $P_\ell^n$ to be the graph with vertex set $[\ell]^n$ where two
vertices $(x_1, \dotsc, x_n), \allowbreak (y_1, \dotsc, y_n)$ are adjacent if
$\sum_{i=1}^n |x_i-y_i| = 1$. We say that a vertex is odd if the sum of its
entries is odd; otherwise, that vertex is even.

\begin{prop} \label{prop:counterexample}

    Let $\ell \ge 3$ be an odd integer. Then there exists a graph $H$ satisfying
    \begin{itemize}
        \item $H$ is isomorphic to an induced subgraph of $P_\ell^m$ for some
            $m$
        \item $|H|$ is a power of $\ell$
        \item for any $n$, it is impossible to partition the vertices of
            $P_\ell^n$ into induced copies of $H$.
    \end{itemize}
    
\end{prop}

\begin{proof}
    Fix an odd integer $\ell \ge 3$ and write $A_n$ and $B_n$ for the number of
    even and odd vertices in $P_\ell^n$, respectively.  For any $n$ the graph
    $P_\ell^n$ contains a Hamiltonian path, which visits vertices of alternating
    parity, so we have $|A_n - B_n| \le 1$.  However, $A_n + B_n = |P_\ell^n| =
    \ell^n$ is odd, so in fact $|A_n - B_n| = 1$. In particular, $A_n \not\equiv
    0 \pmod{\ell}$.

    Now, choose $m$ sufficiently large so that $P_\ell^m$ contains an induced
    connected subgraph on $\ell$ even and $\ell^2 - \ell$ odd vertices. Denote
    this subgraph by $H$. We claim that, for any $n$, it is impossible to
    partition the vertices of $P_\ell^n$ into induced copies of $H$. Indeed,
    each induced copy of $H$ in $P_\ell^n$ contains $\ell$ or $\ell^2-\ell$ even
    vertices.  Therefore, the total number of even vertices covered by such a
    partition would be divisible by $\ell$. However, as we saw previously, the
    number of even vertices in $P_\ell^n$ is not.
\end{proof}

It would be interesting to know if \Cref{thm:main} is particular to the
hypercubes $Q_n$ or if it holds for powers of other graphs as well. More
specifically, let $G,H$ be finite graphs. For any $n$, we define $G^n$ to be the
graph with vertex set $V(G)^n$, where $(u_1, \dotsc, u_n)$ and $(v_1, \dotsc,
v_n)$ are adjacent if and only if there exists an index $i' \in [n]$ such that
$u_i = v_i$ for all $i \neq i'$ and $u_{i'}, v_{i'}$ are adjacent vertices of
$G$. We remark that, with this definition, $Q_n$ is the $n$th power of the path
$P_2$ consisting of a single edge.  What are the natural conditions on $H$ that
would make it reasonable to believe that, for some $n$, $G^n$ can be partitioned
into $H$-sets?  Obviously, $|H|$ has to divide $|G|^n$, so we should assume that
every prime factor of $|G|$ also divides $|H|$. We should also require $H$ to be
isomorphic to an induced sugraph of $G^k$ for some $k$; in fact, we may assume
that $H$ is isomorphic to an induced sugraph of $G$ itself.  However, this is
not enough. First, it may still not be possible to cover $G^n$ with copies of
$H$. Moreover, \Cref{prop:counterexample} tells us that even the extra
assumption that $G$ can be covered by copies of $H$ would not be enough. After
examining why $G = Q_n$ works and $G = P_3^n$ does not, we see that
\Cref{obs:rpart} breaks down because $P_3^n$ is not vertex-transitive. We
conjecture that \Cref{thm:main} holds whenever we replace $Q_n$ by another
vertex-transitive graph.

\begin{conj} \label{conj:any-graph}
    Let $G$ be a finite vertex-transitive graph and let $H$ be an induced
    subgraph of $G$. If every prime factor of $|H|$ divides $|G|$, then there
    exists a positive integer $n$ such that $G^n$ can be partitioned into
    induced copies of $H$.
\end{conj}

What happens if instead of partitioning the vertices of $Q_n$ we attempt to
partition the edges? If we want to partition the edge set of $Q_n$ into copies
of a fixed graph $H$, then the obvious necessary divisibility condition is $e(H)
| 2^{n-1} n$, which is satisfied whenever $n$ is a multiple of $e(H)$.
Therefore, as long as $H$ is isomorphic to a subgraph of $Q_k$ for some $k$, we
may expect that such a partition exists for some $n$. Along with I.  Leader and
T.S. Tan we make the following conjecture.

\begin{conj} \label{conj:edges}
    Let $H$ be a non-empty subgraph of $Q_k$ for some $k$. Then there exists a
    positive integer $n$ such that the edges of $Q_n$ can be covered by
    edge-disjoint copies of $H$ (the copies of $H$ are not required to be
    induced).
\end{conj}

It seems to be difficult to prove \Cref{conj:edges} even in very special cases,
when we choose $H$ to be a fairly simple graph. For example, we do not know if
the conjecture is true when $H$ is $Q_k$ with one edge removed.

On the other hand, the case when $H$ is a path is well understood. Indeed,
the edges of $Q_n$ can be partitioned into antipodal paths of the form $(x_1,
x_2, \dotsc, x_n) \to (1 - x_1, x_2, \dotsc, x_n) \to (1 - x_1, 1 - x_2, \dotsc,
x_n) \to \dotsb \to (1 - x_1, 1 - x_2, \dotsc, 1 - x_n)$ with $x_1 + \dotsb +
x_n$ even. Therefore, $E(Q_n)$ can be partitioned into copies of $P_{k+1}$ whenever
$n$ is a multiple of $k$. Moreover, for odd $n$, Erde~\cite{Erde2014} and
Anick and Ramras~\cite{Anick2015} independently determined exactly when $E(Q_n)$
can be partitioned into copies of $P_{k+1}$: this can be done if and only if $k
\le n$ and $k | 2^{n-1}n$. For even $n$ not everything is known yet. Erde
conjectured that in this case the obviously necessary conditions $k \le 2^n$
and $k | 2^{n-1}n$ are sufficient.

\section{Acknowledgements}

We thank Shoham Letzter for discussions which resulted in
\Cref{prop:counterexample}.

\bibliography{hypercube}

\end{document}